\documentclass[11pt,a4paper]{amsart}

\usepackage{geometry}
\pdfoutput=1
\usepackage{amsmath,amsfonts,amssymb,amsthm}
\usepackage{xcolor}
\usepackage{todonotes}
\usepackage{hyperref}
\hypersetup{
	pagebackref  = true,
	colorlinks   = true,
	pdfauthor    = {Vitezslav Kala},
	pdftitle     = {},
	pdfkeywords  = {},
	pdfcreator   = {Pdflatex},
	pdfpagemode  = UseNone,
	pdfstartview = FitH
}
\usepackage{pdfpages}
\usepackage[normalem]{ulem}
\usepackage{url}
\usepackage{enumerate}
\usepackage{verbatim}
\usepackage{mathrsfs}
\usepackage{csquotes}
\usepackage{mathtools}
\usepackage{lipsum}
\theoremstyle{plain}
\newtheorem{theorem}{Theorem}
\newtheorem{prop}[theorem]{Proposition}
\newtheorem{conj}[theorem]{Conjecture}

\newtheorem*{theorem*}{Theorem}

\newtheorem*{lemma*}{Lemma}
\theoremstyle{definition}

\newcommand{\R}{\mathbb{R}}

\newcommand{\Z}{\mathbb{Z}}
\newcommand{\Q}{\mathbb{Q}}

\newcommand{\co}{\mathcal{O}}
\newcommand{\wt}{\widetilde}

\newcommand{\disc}{\mathrm{disc}}

\DeclareMathOperator\Gal{Gal}

\usepackage{tikz}

\DeclareMathOperator\Tr{Tr}

\DeclarePairedDelimiterX\set[2]\lbrace\rbrace{\,#1\mathclose{}:\mathopen{}#2\,}

\title{Number fields
	without universal quadratic forms\\
	of small rank exist in most degrees}

\author{V\' \i t\v ezslav Kala}
\address{Charles University, Faculty of Mathematics and Physics, Department of Algebra, Sokolov\-sk\' a 49/83, 18675 Praha~8, Czech Republic}
\email{kala@karlin.mff.cuni.cz, vita.kala@gmail.com}

\thanks{
	The author was supported by Czech Science Foundation GA\v CR, grant 21-00420M, and by Charles University,  projects PRIMUS/20/SCI/002 and UNCE/SCI/022}
\begin{document}
	
\maketitle

\begin{abstract}
	We prove that in each degree divisible by 2 or 3, there are infinitely many totally real number fields that require universal quadratic forms to have arbitrarily large rank.
\end{abstract}

\section{Introduction}\label{s:1}

In 1770 Lagrange proved that every positive integer is the sum of four squares, opening up the study of universal quadratic forms. These were then first investigated over the integers $\Z$, leading to the celebrated 15- and 290-theorems \cite{Bh, BH}, and also over number fields, starting with Maa\ss\ \cite{Ma} and Siegel \cite{Si3} in the 1940s. To be precise, let $\co_F$ be the ring of integers in a totally real number field $F$. A totally positive quadratic form $Q$ with $\co_F$-coefficients is \textit{universal} over $F$ if it represents all the totally positive elements of $\co_F$.

Universal forms exist over every $F$ thanks to the weak local-global principle \cite{HKK}. Of particular interest is thus the smallest possible rank $m'(F)$ of a universal form over $F$. 
E.g., we have $m'(\Q)=4$ by the four square theorem. Among real quadratic fields, $m'(\Q(\sqrt D))=3$ for $D=2,3,5$, and these are the only real quadratic fields that admit a ternary universal form that is moreover classical (i.e., has all its cross-terms divisible by $2$) \cite{CKR}. 
This provides interesting evidence towards Kitaoka's conjecture that there are only finitely many number fields $F$ with $m'(F)=3$.

Further, the ranks $m'(\Q(\sqrt D))$ can be arbitrarily large \cite{BK}, \cite{Ka}, ditto for multiquadratic fields of a given degree \cite{KS}. Despite a number of other exciting results obtained in the last 25 years \cite{BK2, CL+, EK, KT, KY, Ki2, KKP, KTZ, Ya},
ranks of universal forms over number fields, especially of higher degree, remain mysterious.

The aim of this short note is to extend the previous special results on unbounded ranks $m'(F)$ to number fields of most degrees:

\begin{theorem}\label{main theorem} Let $d,m$ be positive integers such that $d$ is divisible by $2$ or $3$.
	Then there are infinitely many totally real number fields $F$ of degree $[F:\Q]=d$ over which every  universal quadratic form has rank at least $m$.
\end{theorem}

If $d=2$, this was proved by the author \cite[Theorem 1.1]{Ka}. The key idea was to use continued fractions to construct quadratic fields that have many indecomposable elements, which are hard to represent by a quadratic form.
Constructing such elements in higher degrees is more difficult, nevertheless, the author and Svoboda \cite[Theorem 1]{KS} extended the result to all degrees $d=2^h$. In the cubic case $d=3$, this theorem was proved by Yatsyna \cite[Theorem]{Ya} using interlacing polynomials and elements of trace one.

\medskip

Our argument will use Schur's trace bound \cite{Sch} to show a general Theorem \ref{general theorem}: in certain cases, suitable elements from a cyclic number field $L$ force a quadratic form that represents them to have many variables, even in an overfield. We will then prove Theorem \ref{main theorem} by choosing $L$ to be a real quadratic, or simplest cubic \cite{Sh} number field.

Theorem \ref{main theorem} also holds for quadratic lattices that are not necessarily free; in fact, we will formulate the rest of the article in lattice-theoretic language.
Also note that we do not assume the quadratic forms to be classical, although this is a very common assumption and there are only very few results available without it (e.g., \cite{De}).

\medskip

Finally, at the end of the paper we will observe that Theorem \ref{main theorem} for \textit{all} degrees $d>1$ would follow if we knew that: In each prime degree $p\geq 5$, there are infinitely many cyclic totally real number fields that have a power integral basis and units of all signatures.

\section*{Acknowledgments}

I thank Giacomo Cherubini, Ari Shnidman, B\' ara T\'\i \v zkov\' a, and Pavlo Yatsyna for our interesting and helpful discussions.

\section{Preliminaries}\label{s:2}

Let $F$ be a totally real number field of degree $[F:\Q]=N$ over $\Q$, i.e., there are $N$ real embeddings $\sigma_1,\dots,\sigma_N:F\rightarrow \R$. We denote the ring of algebraic integers $\co_F$. An element $\alpha\in F$ is \textit{totally positive} (denoted $\alpha\succ0$) if $\sigma_i(\alpha)>0$ for all $1\leq i\leq N$. 
Further, $\alpha\succeq\beta$ if $\alpha-\beta\succ 0$ or $\alpha=\beta$.
The set of all totally positive algebraic integers is $\co_F^+$.

For $\alpha\in F$ we have its \textit{trace} $\Tr_{F/\Q}(\alpha)=\sum_{1\leq i\leq N}\sigma_i(\alpha)$ and \textit{discriminant} $\Delta_{F/\Q}(\alpha)$, which is the square of the determinant of the matrix $(\sigma_i(\alpha^{j-1}))_{1\leq i,j\leq N}$.
The \textit{discriminant of $F$}, i.e., the discriminant of an integral basis for $\co_F$, will be denoted $\disc_F$. We have $\disc_F\mid\Delta_{F/\Q}(\alpha)$ for each $\alpha\in\co_F$.

\medskip

A \textit{totally positive quadratic $\co_F$-lattice of rank $r$} (an \textit{$\co_F$-lattice} for short) is a pair $(\Lambda,Q)$, where $\Lambda$ is a finitely generated $\co_F$-submodule of $F^r$ such that $\Lambda F=F^r$, $Q:F^r\rightarrow F$ is a quadratic form, and $Q(v)\in\co_F^+$ for all $v\in \Lambda, v\neq 0$. 
We also have the attached symmetric bilinear form $B(v,w)=(Q(v+w)-Q(v)-Q(w))/2$.
An $\co_F$-lattice $(\Lambda,Q)$ is \textit{universal} (over $F$) if for each $\alpha\in\co_F^+$ there is $v\in\Lambda$ with $Q(v)=\alpha$.

Let $m(F)$ denote the minimal rank of a universal $\co_F$-lattice. Note that to each quadratic form $Q$ (as considered in the Introduction) corresponds the $\co_F$-lattice $(\co_F^r,Q)$, and so $m(F)\leq m'(F)$ (it is interesting to note that no example of strict inequality is known here).

\medskip

Take non-zero vectors $v_1,\dots,v_n\in\Lambda$. The corresponding \textit{Gram matrix} is the $n\times n$ matrix 
$A=(B(v_i,v_j))_{1\leq i,j\leq n}$.
Note that we have $B(v_i,v_i)=Q(v_i)=a_i$ and $B(v_i,v_j)=b_{ij}/2$ for all $i\neq j$ and suitable $a_i, b_{ij}\in\co_{F}$.
As the lattice $\Lambda$ is totally positive, $a_i\succ 0$ and
we have a version of \textit{Cauchy--Schwarz inequality}
$4a_ia_j \succeq b_{ij}^2$ for all $i\neq j$, equivalently, $Q(v_i)Q(v_j)\succeq B(v_i,v_j)^2$ (this quickly follows from the positive-definiteness of the quadratic form $\sigma_h(Q)$ on $\R^r$ for $h=1,\dots,N$).

Also note that the rank of $A$ (as a matrix over the field $F$) is at most the rank $r$ of the lattice $\Lambda$
(for the rank of $A$ equals the rank of the $\co_F$-sublattice of $\Lambda$ spanned by $v_1,\dots,v_n$). For more background on quadratic lattices, see \cite{O1}.

\medskip

Further, we will crucially use the following lower bound due to Schur.

\begin{prop}[{\cite[\S 2.II.]{Sch}}]\label{prop:schur}
Let $F$ be a totally real number field of degree $[F:\Q]=N$.

If $\beta\in\co_F$, then
	$$\Tr_{F/\Q} \beta^2\geq c_{N}\Delta_{F/\Q}(\beta)^{2/(N^2-N)},\text{ where }c_{N}=\frac{N^2-N}{\left(2^2\cdot3^3\cdot 4^4\cdots(N-1)^{N-1}\cdot N^N\right)^{2/(N^2-N)}}.$$
\end{prop}

\begin{proof}
	Schur's bound \cite[\S 2.II.]{Sch} states that if $x_1,\dots,x_N$ are real numbers such that $x_1^2+\dots+x_N^2\leq 1$, then the discriminant $\prod_{1\leq i<j\leq N}(x_i-x_j)^2\leq c_N^{-(N^2-N)/2}$. Setting $x_i=\sigma_i(\beta)/(\Tr_{F/\Q} \beta^2)^{1/2}$ gives the inequality we need.	
\end{proof}

For an integer $k\geq 2$, we will denote $S_k$ and $A_k$ the symmetric and alternating groups on the set $\{1,2,\dots,k\}$. The cyclic group of order $k$ is $C_k$ (considered multiplicatively).

We will work with extensions of a given number field by an $S_k$-number field, whose existence is given by the following proposition.

\begin{prop}\label{prop:prim}
	Let $D, X, k\geq 2$.
	There are infinitely many totally real number fields $K$ of degree $[K:\Q]=k$ whose 
	discriminant $\disc_K>X$ is coprime with $D$ and whose Galois closure $\widetilde K$ has Galois group $\Gal(\wt K/\Q)\simeq S_k$.
\end{prop}

\begin{proof}
	This is well-known. The most straightforward proof is probably using Hilbert's irreducibility theorem 
	(see, e.g., \cite[Theorem 4.2.3]{Kaly}). 
	
	Much more strongly, Kedlaya \cite[Theorem 1.1]{Ke} proved that one can even impose the additional condition that the $\disc_K$ is squarefree.	
	Further, Bhargava, Shankar, and Wang  \cite{BSW} proved that the polynomials $f(x)=x^k+a_1x^{k-1}+\dots+a_{k-1}x+a_k$, whose rupture field $K$ has the required properties (including squarefree $\disc_K$), have positive density	
	 when ordered by $\max\{|a_i|^{1/i}\mid 1\leq i\leq k\}$.
\end{proof}

\section{The Proof}\label{s:gen}

To prove Theorem \ref{main theorem}, we will use the following general theorem that we will then apply to suitable fields $L$ (of degrees $\ell=2,3$).

\begin{theorem}\label{general theorem} Let $k,\ell,m,n$ be positive integers such that $k=3$ or $k\geq 5$. 
	
	Assume that there is a totally real Galois number field $L$ of degree
	$[L:\Q]=\ell$ whose Galois group is $\Gal(L/\Q)\simeq C_\ell$ and that contains elements $a_1,\dots,a_n\in\co_L^+$ such that if an $\co_L$-lattice represents $a_1,\dots,a_n$, then it has rank $\geq m$.
	
	There is $B>0$ (depending on $k,\ell,L,a_i$) with the following property:
	
	For every totally real number field $K$ of degree $[K:\Q]=k$ whose discriminant $\disc_K>B$ is coprime with $\disc_L$ and whose Galois closure $\widetilde K$ has Galois group $\Gal(\wt K/\Q)\simeq S_k$, we have $[KL:\Q]=k\ell$ and
	$$m(KL)\geq m.$$
\end{theorem}

\begin{proof}
	Let $L,K,\wt K$ be as in the statement (with $B$ to be specified later).
	
	As $\disc_K$ and $\disc_L$ are coprime, we have $K\cap L=\Q$ (for $\disc_{K\cap L}$ is a common divisor of $\disc_K$ and $\disc_L$ by the formula for the discriminant of a tower of number fields \cite[Corollary III.2.10]{Ne}).
	Thus $[KL:\Q]=k\ell$.
	Let us use Galois theory to describe all subfields $\Q\subset M\subset KL$ (without giving references for all the theorems that we use -- see any good textbook on Abstract Algebra).
	
	\medskip
	
	First, consider $H=\wt K\cap L$. $H$ is a subfield of $L$, and so $\Gal(H/\Q)\simeq C_t$ for some $t\mid\ell$, and $\Gal(\wt K/H)$ is a normal subgroup of $\Gal(\wt K/\Q)\simeq S_k$. The only such subgroups are $S_k$ and $A_k$ (as $k\neq 4$), and so correspondingly, $H=\Q$ or $H=\Q(\sqrt{\disc_K})$. But the latter case is impossible, as
	$\disc_K$ and $\disc_L$ are coprime, and so $\sqrt{\disc_K}\not\in L$.
	
	Thus $\wt K\cap L=\Q$, and so $\wt KL$ is Galois with $$\Gal(\wt KL/\Q)\simeq S_k\times C_\ell.$$
	Further, 
	$$\Gal(\wt KL/KL)\simeq S_{k-1}\times \{1\}$$
	(where we view $S_{k-1}\subset S_k$ as the subgroup of permutations that fix the element $k$, e.g.).
	
	Thus by Galois correspondence, the fields $\Q\subset M\subset KL$ correspond to subgroups
	$$S_k\times C_\ell\supset G\supset S_{k-1}\times \{1\}.$$
	
	\medskip
	
	We claim that for each such subgroup, we have $G\subset S_{k-1}\times C_\ell$ or $G\supset S_k\times \{1\}$.
	 
	For if $G\not\subset S_{k-1}\times C_\ell$, then there is an element $(\sigma,u)\in G$ with $\sigma\not\in S_{k-1}$. By multiplying it by a suitable element $(\tau,1)\in S_{k-1}\times \{1\}\subset G$, we obtain	
	$((1k),u)\in G$  (where $(ij)$ denotes the 2-cycle  in $S_k$ that exchanges $i,j$). 
	
	If the order $o$ of $u$ in $C_\ell$ is odd, then 
	$((1k),1)=((1k),u)^o\in G$. If $o$ is even, then also 
	$$G\ni ((12),1)[((1k),u)((12),1)((1k),u)((1k),u)^{o-2}]((12),1)=((12),1)((2k),1)((12),1)=((1k),1).$$
	Thus $G\supset S_k\times \{1\}$ (as $((1k),1)\in G$ and $G\supset S_{k-1}\times \{1\}$), as we wanted to show.
	
	\medskip
	
	Correspondingly, each intermediate field $\Q\subset M\subset KL$ satisfies 
	$$M\supset K\text{ or } M\subset L.$$
	
	\medskip
	
	Let us finally specify that  
	$$\disc_K>B\text{ for }
	B=\text{max}_M\left(\left(\frac{keT}{\ell c_{ke}}\right)^{(k^2e-k)/2}
	\right),$$ where
	the maximum is taken over all fields $M$ such that $K\subset M\subset KL$,
	$e=[M:K]$,	
	$T=4\max\{\Tr_{L/\Q}(a_ia_j)\mid 1\leq i<j\leq n\}$
	and $c_{ke}$ are the constants from Proposition \ref{prop:schur}. 
	
	\medskip
	
	Let $(\Lambda,Q)$ be a universal $\co_{KL}$-lattice. As $L\subset KL$, the lattice $\Lambda$ represents all the elements $a_1,\dots,a_n\in\co_L^+$; fix vectors $v_i\in \Lambda$ such that $Q(v_i)=a_i$. 
	We will show that $\Lambda$ has rank $\geq m$ by showing that the Gram matrix $(B(v_i,v_j))_{1\leq i,j\leq n}$ corresponding to the vectors $v_i$ has rank $\geq m$.
	
	We have $B(v_i,v_i)=Q(v_i)=a_i$ and let $B(v_i,v_j)=b_{ij}/2$ for all $i\neq j$ and suitable $b_{ij}\in\co_{KL}$.
	We will now show that $b_{ij}\in L$ for all $i\neq j$. 
	
	Assume that this is not the case for some $i\neq j$ and let $M=\Q(b_{ij})$. By the description of possible fields $\Q\subset M\subset KL$ obtained above, we have that $M\supset K$. Let $[M:K]=e$; then $\disc_M\geq\disc_K^e$ (again by the formula for the discriminant of a tower).
	
	As the $\co_{KL}$-lattice $\Lambda$ is totally positive, we have the Cauchy--Schwartz inequality  $4a_ia_j \succeq b_{ij}^2$ (see Section~\ref{s:2}).
	
	Taking traces and applying Proposition \ref{prop:schur} for the field $M$ of degree $[M:\Q]=ke$, 
	we get
	$$kT\geq \Tr_{KL/\Q}(4a_ia_j)\geq \Tr_{KL/\Q}(b_{ij}^2)
	= \frac \ell e \Tr_{M/\Q}(b_{ij}^2)	
	\geq \frac \ell e c_{ke}\Delta_{M/\Q}(b_{ij})^{2/((ke)^2-ke)}.$$
	
	As $b_{ij}$ does not lie in a proper subfield of $M$, we have $\Delta_{M/\Q}(b_{ij})\neq 0$, and so $\Delta_{M/\Q}(b_{ij})\geq \disc_M\geq \disc_K^e>B^e$. Thus
	$$kT>\frac \ell e c_{ke}B^{2/(k^2e-k)},$$
	contradicting the choice of $B$.
	
	We proved that $b_{ij}\in L$ for all $i\neq j$. 
	
	\medskip
	
	Therefore all the entries of the Gram matrix $(B(v_i,v_j))_{1\leq i,j\leq n}$ lie in $L$, and so this matrix corresponds to an $\co_L$-lattice $\Lambda'$ that represents all the elements $a_1,\dots,a_n$ over $L$.
	By the assumption of the theorem, every such lattice has rank $\geq m$.
	
	Accordingly, the Gram matrix $(B(v_i,v_j))_{1\leq i,j\leq n}$ has rank $\geq m$, which 
	finally implies that the rank of $\Lambda$ is also $\geq m$.	
\end{proof}

We can now finally use the preceding result to prove our main theorem.

\begin{theorem*}[Theorem \ref{main theorem}, lattice-theoretic formulation] Let $d,m$ be positive integers such that $d$ is divisible by $2$ or $3$.
	Then there are infinitely many totally real number fields $F$ of degree $[F:\Q]=d$ over which every totally positive universal quadratic $\co_F$-lattice has rank at least $m$, i.e., $m(F)\geq m$.
\end{theorem*}

\begin{proof}%[Proof of Theorem \ref{main theorem}]
	For $d=2,4,8$, this was proved by Kala--Svoboda \cite[Theorem 1]{KS}, and for $d=3$ by Yatsyna \cite[Theorem 5]{Ya}.
	
	\medskip
	
	If $d=6$ or $d\geq 10$ is even, choose $\ell=2$ and $k=d/2$; we have $k=3$ or $k\geq 5$.
	
	By \cite[Section 4]{Ka}, there are (infinitely many) real quadratic fields $L=\Q(\sqrt D)$ that contain $n=m$ elements ($\alpha_1,\alpha_3,\dots,\alpha_{2M+1}$ in the notation of \cite{Ka} for $M=m-1$) such that their corresponding Gram matrix is diagonal by \cite[Proposition 4.1]{Ka}, and so every $\co_K$-lattice that represents these elements has rank $\geq m$.
	
	By Proposition \ref{prop:prim}, there are infinitely many fields $K$ of degree $k=d/2$ with the properties required by Theorem \ref{general theorem}, and so for each of them we have $[KL:\Q]=d$ and $m(KL)\geq m$, as needed.
	
	\medskip 
	
	If $d=9$ or $d\geq 15$ is divisible by $3$, choose $\ell=3$ and $k=d/3$; again $k=3$ or $k\geq 5$.
	
	Let $n=\max(9m^2,240)$ and let us consider Shanks' \textit{simplest cubic fields} \cite{Sh} $L=\Q(\rho)$ (where $\rho$ is a root of the polynomial $x^3-ax^2-(a+3)x-1$ for some $a\in\Z_{\geq -1}$). 
	Each simplest cubic field $L$ is Galois with $\Gal(L/\Q)\simeq C_3$.
	
	Kala--Tinkov\' a \cite[Subsection 7.2]{KT} proved that there are (infinitely many) such fields $L$ that contain $n$ elements $a_1, a_2, \dots, a_n\in\mathcal O_L^+$ and an element $\delta\succ 0$ in the codifferent $\{\alpha\in L\mid \Tr_{L/\Q}(\alpha\beta)\in\Z\text{ for all }\beta\in\co_L\}$ with $\Tr_{L/\Q}(\delta a_i)=1$ for all $i$. 
	By \cite[Subsection 7.2 and proof of Proposition 7.4]{KT}, if an $\co_L$-lattice represents all the elements $a_1,\dots,a_n$, then it has rank $\geq \sqrt {n}/3=m$.
	
	It again just remains to use Proposition \ref{prop:prim} and Theorem \ref{general theorem}.
	
	\medskip

	This covers all the positive integers $d$ that are divisible by $2$ or $3$, finishing the proof.
\end{proof}

It is unfortunate that Theorem \ref{general theorem} assumes the existence of suitable elements $a_1,\dots,a_n$, instead of directly claiming that $m(KL)\geq m(L)$.
Such a cleaner version of Theorem \ref{general theorem} would easily follow
from the following folklore result, whose proof unfortunately does not appear in the literature, so we state it here only as a conjecture:

\begin{conj}\label{prop:weak 290}
	Let $F$ be a totally real number field. There is a finite subset $S\subset \co_F^+$ such that if an $\co_F$-lattice represents all the elements of $S$, then it is universal.	
\end{conj}

Over the rationals $\co_F=\Z$ this is of course well-known, as it is just a weak version of the famous 290-theorem \cite{BH}.
Even more generally, Kim--Kim--Oh \cite{KKO} proved a similar result for representations of quadratic forms by quadratic forms (over $\Z$), and remarked that their theorem should also hold over number fields.
It seems that this is indeed the case; unfortunately the general proof would be somewhat lengthy, and so we chose not to include it in this short note.
A proof of a more general result than Conjecture \ref{prop:weak 290} should eventually appear in \cite{S}, but is unpublished yet.

Also note that Theorem \ref{general theorem} likely also holds with different Galois groups than $C_\ell$ and $S_k$ (for example, \cite{KS} dealt with multiquadratic fields $F$, i.e., $\Gal(F/\Q)\simeq C_2^h$). However, the present formulation is sufficient for the proof of our main Theorem \ref{main theorem}, and a more general statement probably would not bring more clarity.

\medskip

To conclude, let us briefly discuss a way of extending Theorem \ref{main theorem} to all degrees $d>1$. 
Let $p\geq 5$ be a prime, and assume that there are infinitely many totally real cyclic fields $L$ of degree $p$ that have a power integral basis and units of all signatures. Each such field is primitive (i.e., has no proper subfield) and satisfies Yatsyna's ``Condition (A)'' \cite{Ya}. 
For each $m$, by \cite[proof of Theorem 4]{Ya} all but finitely many of these fields contain elements $a_1,\dots,a_n$ such that every quadratic $\co_L$-lattice representing these elements has rank $\geq m$. Thus we can apply Theorem \ref{general theorem} to any of these infinitely many suitable fields to conclude that $m(KL)\geq m$ for some field $K$ with $[KL:\Q]=kp, k=3$ or $k\geq 5$. Knowing this for all primes $p\geq 5$, together with our Theorem \ref{main theorem}, we cover all the degrees $d>1$.

Unfortunately, while there probably indeed are infinitely many fields satisfying our assumption, this appears not to have been proven yet (and is perhaps hard).


\begin{thebibliography}{BHM}
	
		
	\bibitem[Bh]{Bh} M. Bhargava, \emph{On the Conway-Schneeberger Fifteen Theorem},  Contemp. Math. \textbf{272} (1999), 27--37  
	
	\bibitem[BH]{BH} M. Bhargava, J. Hanke, \emph{Universal quadratic forms and the 290-theorem}, preprint
	
	\bibitem[BSW]{BSW} M. Bhargava, A. Shankar, X. Wang, \textit{Squarefree values of polynomial discriminants I}, preprint,
	\url{https://arxiv.org/abs/1611.09806}
	
	\bibitem[BK1]{BK} V. Blomer, V. Kala, \emph{Number fields without universal $n$-ary quadratic forms}, Math. Proc. Cambridge Philos. Soc. \textbf{159} (2015), 239--252
	
	\bibitem[BK2]{BK2} V. Blomer, V. Kala, \emph{On the rank of universal quadratic forms over real quadratic fields}, Doc. Math. \textbf{23} (2018), 15--34
	
	
	\bibitem[CL+]{CL+} M. \v Cech, D. Lachman, J. Svoboda, M. Tinkov\' a, K. Zemkov\' a, \emph{Universal quadratic forms and indecomposables over biquadratic fields}, Math. Nachr. \textbf{292} (2019), 540--555
	
	
	\bibitem[CKR]{CKR} W. K. Chan, M.-H. Kim, S. Raghavan, \emph{Ternary universal integral quadratic forms}, Japan. J. Math. \textbf{22} (1996), 263--273
	
	\bibitem[De]{De} J. I. Deutsch, \emph{Universality of a non-classical integral quadratic form over $\mathbb Q(\sqrt 5)$}, Acta Arith. \textbf{136} (2009), 229--242
	
	\bibitem[EK]{EK} A. G.
	Earnest, A. Khosravani, 
	\emph{Universal positive quaternary quadratic lattices over totally real number fields},
	Mathematika \textbf{44} (1997), 342--347
	
	\bibitem[HKK]{HKK} J. S. Hsia, Y. Kitaoka, M. Kneser, \emph{Representations of positive definite
	quadratic forms}, J. Reine Angew. Math. \textbf{301} (1978), 132--141

	
	\bibitem[Ka]{Ka} V. Kala, \emph{Universal quadratic forms and elements of small norm in real quadratic fields}, Bull. Aust. Math. Soc. \textbf{94} (2016), 7--14
	
	\bibitem[KS]{KS} V. Kala, J. Svoboda, \emph{Universal quadratic forms over multiquadratic fields}, Ramanujan J. \textbf{48} (2019), 151--157
	
	\bibitem[KT]{KT} V. Kala, M. Tinkov\' a, \emph{Universal quadratic forms, small norms and traces in families of number fields}, preprint, \url{https://arxiv.org/abs/2005.12312} 
	
	\bibitem[KY]{KY} V. Kala, P. Yatsyna,  \textit{Lifting problem for universal quadratic forms},  Adv. Math. 377 (2021), 107497, 24 pp.
	
	\bibitem[Kal]{Kaly} S. Kalyanswamy, \textit{Inverse Galois Problem for Totally Real Number Fields}, Cornell University Mathematics Department Senior Thesis, 2012, \url{http://pi.math.cornell.edu/files/Research/SeniorTheses/kalyanswamyThesis.pdf}
	
	\bibitem[Ke]{Ke} K. S. Kedlaya, \textit{A construction of polynomials with squarefree discriminants},
	Proc. Amer. Math. Soc. \textbf{140} (2012), 3025--3033
	
	
	\bibitem[Ki]{Ki2} B. M. Kim, \emph{Universal octonary diagonal forms over some real quadratic fields}, Commentarii Math. Helv. \textbf{75} (2000), 410--414
	
	\bibitem[KKO]{KKO}	B. M. Kim, M.-H. Kim, B.-K. Oh, \textit{A finiteness theorem for representability of quadratic forms by forms}, J. Reine Angew. Math. \textbf{581} (2005), 23--30
	
	
	\bibitem[KKP]{KKP}	B. M. Kim, M.-H. Kim, D. Park, \textit{Real quadratic fields admitting universal lattices of rank 7}, preprint, \url{https://arxiv.org/abs/2006.15361}
	
	\bibitem[KTZ]{KTZ} J. Kr\' asensk\' y, M. Tinkov\' a, K. Zemkov\' a, \emph{There are no universal ternary quadratic forms over biquadratic fields}, Proc. Edinb. Math. Soc. \textbf{63} (2020), 861--912
	
	\bibitem[Ma]{Ma} H. Maa{\ss}, \emph{\" Uber die Darstellung total positiver Zahlen des K\" orpers $R(\sqrt 5)$
		als Summe von drei Quadraten}, Abh. Math. Sem. Univ. Hamburg \textbf{14} (1941), 185--191
	
	\bibitem[Neu]{Ne} J. Neukirch, \emph{Algebraic number theory},  Springer-Verlag, Berlin, 1999
	
	\bibitem[OM]{O1} O. T. O'Meara, \emph{Introduction to Quadratic Forms}, Springer-Verlag, Berlin, 1973
	
	\bibitem[Sch]{Sch} I. Schur, \textit{\" Uber die Verteilung der Wurzeln bei gewissen algebraischen Gleichungen mit ganzzahligen Koeffizienten} Math. Z. \textbf{1} (1918), 377--402

	\bibitem[Sh]{Sh} D. Shanks, \textit{The simplest cubic number fields}, Math. Comp. \textbf{28}  (1974), 1137--1152

	\bibitem[Si]{Si3} C. L. Siegel, \emph{Sums of $m$-th powers of algebraic integers},  Ann. of Math. \textbf{46} (1945),  313--339 

	\bibitem[Su]{S} L. Sun, \textit{A finiteness theorem for quadratic forms}, preprint
	
	\bibitem[Ya]{Ya} P. Yatsyna, \emph{A lower bound for the rank of a universal
		quadratic form with integer coefficients in a totally real field}, Comment. Math. Helvet. 
	\textbf{94} (2019), 221--239
	
\end{thebibliography}
\end{document}